\documentclass[12pt]{amsart}
\usepackage[latin1]{inputenc}
\usepackage{amssymb, amsmath, amscd, mathrsfs,marvosym, psfrag} 

\usepackage{pstricks}

\input xy
\xyoption{all}
\DeclareMathAlphabet{\mathpzc}{OT1}{pzc}{m}{it}

\newtheorem{theorem}{Theorem}[section]

\newtheorem{proposition}[theorem]{Proposition}
\newtheorem{corollary}[theorem]{Corollary}

\newtheorem{lemma}[theorem]{Lemma}

\theoremstyle{definition}
\newtheorem{definition}[theorem]{Definition}

\theoremstyle{remark}

\def\varle{\leqslant}

\newcommand{\CI}{{\mathcal I}}
\newcommand{\CJ}{{\mathcal J}}

\newcommand{\CO}{{\mathcal O}}

\newcommand{\CW}{{\mathcal W}}

\newcommand{\CZ}{{\mathcal Z}}

\newcommand{\fh}{{{\mathfrak h}}} 
 
\newcommand{\fp}{{{\mathfrak p}}} 
\newcommand{\fg}{{{\mathfrak g}}} 
\newcommand{\fb}{{{\mathfrak b}}} 
 
\newcommand{\fl}{{{\mathfrak l}}} 
 
\newcommand{\fm}{{{\mathfrak m}}}

\newcommand{\fhd}{\fh^\star}

\newcommand{\hCW}{{\widehat\CW}}

\newcommand{\hfh}{{\widehat\fh}}
\newcommand{\hfg}{{\widehat\fg}}
\newcommand{\hfb}{{\widehat\fb}}

\newcommand{\hS}{{\widehat S}}

\newcommand{\hR}{{\widehat R}}

\newcommand{\hfhd}{\widehat{\fh}^\star}

\newcommand{\tS}{{\widetilde{S}}}
\newcommand{\tQ}{{\widetilde{Q}}}

\newcommand{\tfg}{{\widetilde{\fg}}}

\newcommand{\DC}{{\mathbb C}}

\newcommand{\DZ}{{\mathbb Z}}
\newcommand{\DK}{{\mathbb K}}

\newcommand{\DN}{{\mathbb N}}

\newcommand{\Ext}{{\operatorname{Ext}}}
\newcommand{\Hom}{{\operatorname{Hom}}}

\newcommand{\catmod}{{\operatorname{-mod}}}

\DeclareMathOperator{\cha}{\mathrm{ch}}

\newcommand{\im}{{\operatorname{im}}}

\newcommand{\rk}{{{\operatorname{rk}}}}

\newcommand{\ol}{\overline}
\newcommand{\ul}{\underline}

\newcommand{\id}{{\operatorname{id}}}
\newcommand{\res}{{\operatorname{res}}}
\newcommand{\re}{{\operatorname{re}}}
\newcommand{\crit}{{\operatorname{crit}}}

\newcommand{\GL}{{\operatorname{GL}}}

\newcommand{\rCO}{{\ol\CO}}
\newcommand{\rP}{{\ol P}}
\newcommand{\rnabla}{{\ul\nabla}}
\newcommand{\rDelta}{{\ol\Delta}}
\newcommand{\comment}[1]{}
\begin{document}

\pagenumbering{arabic}
\title[Restricted Linkage principle]{The linkage principle for restricted critical level representations
  of affine Kac--Moody algebras}
\author{Tomoyuki Arakawa, Peter Fiebig}
\thanks{T.A. is partially supported 
by the JSPS Grant-in-Aid  for Scientific Research (B)
No.\ 20340007.\\
\indent P.F. is partially supported by a grant of the Landesstiftung Baden--W\"urttemberg and by the DFG-Schwerpunkt 1388}
\begin{abstract} We study the restricted category $\CO$ for an affine Kac--Moody algebra at the critical level. In particular, we prove the first part of the Feigin--Frenkel conjecture: the linkage principle for restricted Verma modules. Moreover, we prove a version of the BGG-reciprocity principle and we determine the  block decomposition of the restricted category $\CO$. For the proofs we need a deformed version of the classical structures, so we mostly work in a relative setting. 
\end{abstract}

\maketitle

\section{Introduction}
The representation theory of an affine Kac--Moody algebra at the
critical level is of central importance in the approach towards the
geometric Langlands program that was proposed by Edward Frenkel and
Dennis Gaitsgory in \cite{FG06}. While there is already a good
knowledge on the connection between critical level representations and
the geometry of the associated affine Grassmannian, central problems,
for example the determination of the critical simple highest weight
characters, 
still remain open. In this paper we continue our approach towards a
description of the critical level category $\CO$, started in
the paper \cite{AF08}.

Let $\hfg$ be the affine Kac--Moody algebra associated to a finite dimensional, simple complex Lie algebra $\fg$ (for the specialists we point out that we add the derivation operator to the centrally extended loop algebra). We study the corresponding highest weight category  $\CO$.

The Lie algebra $\hfg$ has a one dimensional center and we let $K\in\hfg$ be one of its generators. The center acts semisimply on each object of $\CO$, so  $\CO$ decomposes according to the eigenvalue of the action of $K$. We say that an object $M$ of $\CO$ has level $k\in\DC$ if $K$ acts on $M$ as multiplication with $k$, and we let $\CO_k$ be the full subcategory  of $\CO$ that consists of all modules of level $k$.   There is one special value, $k=\crit$, which is called the {\em critical} level. It is the level of the simple highest weight module $L(- \rho)$, where $\rho$ is a Weyl vector, i.e.~a vector that takes the value $1$ on each simple affine coroot. In the usual normalization (see Section \ref{Sec-affKM}) it is $\crit=-h^\vee$, where $h^\vee$ is the dual Coxeter number.

For all levels $k\ne \crit$ the categorical structure of $\CO_k$ is well-known and admits a description in terms of the affine Hecke algebra associated to $\hfg$, in analogy to the case of the category $\CO$ for a finite dimensional simple complex Lie algebra  (cf.~ \cite{Fie06}). However, for $k=\crit$ the structure  changes drastically. In fact, Lusztig anticipated in his ICM address in 1990 that the representation theory at the critical level resembles the representation theory of a small quantum group or  a modular Lie algebra (cf.~ \cite{Lus91}). In particular, it should not be the affine Hecke algebra that governs the structure of $\CO_\crit$, but its periodic module. The Feigin--Frenkel conjecture on the simple critical characters (cf.~ \cite{AF08}) points in this direction as well.  So one might hope that there is a description of the critical level representation theory that closely resembles the one given for small quantum groups and modular Lie algebras by Andersen, Jantzen and Soergel in \cite{AJS94}. 

The main result in this paper is another step towards such a description (following the paper \cite{AF08}). We prove the restricted linkage principle, i.e.~ we show that a simple module occurs in a restricted Verma module only if their highest weights lie in the same orbit under the associated integral Weyl group. Moreover, we study restricted projective objects,  prove a BGG-reciprocity result and describe the corresponding block decomposition. Our results are in close analogy to the quantum group and the modular case, hence they  strongly support the above conjectures. \subsection{Acknowledgments:} We would like to thank Henning Haahr Andersen and Jens Carsten Jantzen for very motivating and inspiring discussions on the subject of this paper. We would also like to thank the Newton Institute in Cambridge for its hospitality during the program "Algebraic Lie Theory".

\section{Affine Kac--Moody algebras and the deformed category $\CO$}\label{Sec-affKM}

In this section we recall the construction of the deformed category
$\CO$ associated to an affine Kac--Moody algebra. Our main reference for the structure theory is \cite{Kac} and for the deformed representation theory it is
\cite{FieMZ}.

\subsection{Affine Kac--Moody algebras}
We fix a finite dimensional, complex, simple Lie algebra $\fg$ and denote by $\hfg$ the corresponding affine Kac--Moody algebra. As a vector space we have $\hfg=(\fg\otimes_\DC \DC[t,t^{-1}])\oplus\DC K\oplus\DC D$ and the Lie bracket is given by 
\begin{align*}
[K,\hfg] & = 0, \\
[D,x\otimes t^n] &  = n x\otimes t^n, \\
[x\otimes t^m, y\otimes t^n] & = [x,y]\otimes t^{m+n}+
m\delta_{m,-n}\kappa(x,y) K
\end{align*}
for $x,y\in \fg$, $m,n\in \DZ$. Here $\kappa\colon \fg\times\fg\to\DC$ denotes the Killing form for $\fg$.

Let us fix a Borel subalgebra $\fb\subset \fg$ and a Cartan subalgebra $\fh\subset\fg$ inside $\fb$.
The corresponding Cartan and Borel subalgebras of $\hfg$ are
\begin{align*}
\hfh &:=\fh\oplus\DC K\oplus\DC D\\
\hfb  &:= (\fg \otimes_\DC t\DC[t]+\fb\otimes_\DC \DC[t])\oplus \DC
K\oplus\DC D.
\end{align*}

\subsection{Roots of $\hfg$}

The decomposition $\hfh=\fh\oplus\DC K\oplus \DC D$ allows us to embed
$\fhd$ in $\hfhd$ using the map that is dual to the projection $\hfh\to\fh$. 
Let $\delta,\Lambda_0\in\hfhd$ be the elements dual to $D$ and $K$,
resp., with respect to the direct decomposition, so we have
$\delta(\fh\oplus\DC K)=\Lambda_0(\fh\oplus\DC D)=\{0\}$ and
$\delta(D)=\Lambda_0(K)=1$.
 Then $\hfhd=\fhd\oplus\DC \Lambda_0\oplus\DC \delta$.

Let $R\subset\fhd$ be the set of roots of $\fg$ with respect to
$\fh$ and $\fg=\fh\oplus\bigoplus_{\alpha\in R}\fg_\alpha$ the root space decomposition. The set of roots of $\hfg$ with respect to $\hfh$ then is
$\hR=\hR^{\re}\cup\hR^{\im}$, where
\begin{align*}
\hR^{\re}&=\{\alpha+n\delta\mid\alpha\in R, n\in\DZ\},\\
\hR^{\im}&=\{n\delta\mid
n\in\DZ,n\ne 0\}.
\end{align*}
The sets $\hR^{\re}$ and $\hR^{\im}$ are called the sets of real and
of imaginary roots, resp. 
The corresponding root spaces are
\begin{align*}
\hfg_{\alpha+n\delta}&=\fg_\alpha\otimes t^n,\\
\hfg_{n\delta}&=\fh\otimes t^n.
\end{align*}
The positive roots $\hR^+\subset\hR$ are the roots of $\hfb$. Explicitly, we have
$$
\hR^+=\{\alpha+n\delta\mid \alpha\in R, n>0\}\cup \{\alpha\mid
\alpha\in R^+\}\cup\{n\delta\mid n>0\},
$$
where $R^+\subset R$ denotes the roots of $\fb\subset\fg$. We set $\hR^{+,\re}:=\hR^{+}\cap\hR^{\re}$ and $\hR^{+,\im}:=\hR^{+}\cap\hR^{\im}$. We denote
by $\Pi\subset R$ the set of simple roots corresponding to our choice
of $\fb$. The set of
simple affine roots is 
$$
\widehat\Pi=\Pi\cup\{-\gamma+\delta\},
$$
where $\gamma\in R^+$ is the highest root.

\subsection{The Weyl group and the bilinear form}
To any real root $\alpha\in\hR^{\re}$ there is an associated coroot
$\alpha^\vee\in\hfh$ and a reflection  $s_\alpha\colon \hfhd\to\hfhd$ given by
$s_\alpha(\lambda)=\lambda-\langle\lambda,\alpha^\vee\rangle\alpha$.
The affine Weyl group associated to our data is the subgroup $\hCW$ of
$\GL(\hfhd)$ generated by the  $s_\alpha$ with
$\alpha\in\hR^{\re}$. 

We denote by $(\cdot,\cdot)\colon\hfg\times\hfg\to\DC$ the standard bilinear form that is non-degenerate, symmetric and invariant, i.e.~it satisfies
$([x,y],z)=(x,[y,z])$ for $x,y,z\in\hfg$. Its restriction to
$\hfh\times\hfh$ is non-degenerate as well and hence induces a
non-degenerate bilinear form on $\hfhd$ that we denote again by
$(\cdot,\cdot)$. It is explicitly given by the following formulas:
\begin{align*}
(\alpha,\beta) &= \kappa(\alpha,\beta),\\
(\Lambda_0,\fhd\oplus\DC\Lambda_0) &= 0,\\
(\delta,\fhd\oplus\DC\delta) &= 0,\\
(\Lambda_0,\delta) &= 1,
\end{align*}
for $\alpha,\beta\in\fhd$ (here we denote by $\kappa\colon \fhd\times\fhd\to\DC$ the bilinear form induced by the Killing form). Moreover, it is invariant under the
action of $\hCW$, i.e.~for $\lambda,\mu\in\hfhd$ and $w\in\hCW$ we
have
$$
(\lambda,\mu)=(w(\lambda),w(\mu)).
$$

\subsection{The deformed  category $\CO$}

Let $S:=S(\fh)$  and  $\hS:=S(\hfh)$ be the symmetric algebras over the complex vector spaces $\fh$ and $\hfh$. The projection $\hfh\to \fh$ along the decomposition $\hfh=\fh\oplus\DC K\oplus \DC D$ yields an algebra  homomorphism $\hS\to S$. We  think from now on of $S$ as an $\hS$-algebra via this homomorphism. 

Let $A$ be a commutative, associative, noetherian, unital
$S$-algebra. In the following we call such an algebra a {\em
  deformation algebra}. Using the homomorphism $\hS\to S$ from above
we can consider $A$ as an $\hS$-algebra as well. We denote by
$\tau\colon \hfh\to A$ the composition of the canonical map $\hfh\to
\hS$ with the structure homomorphism $\hS\to A$,  $f\mapsto f\cdot 1_A$. Note that $\tau(D)=\tau(K)=0$. 

For any complex Lie algebra $\fl$ we denote by $\fl_A:=\fl\otimes_\DC A$ the $A$-linear Lie algebra obtained from $\fl$ by base change. An  $\fl_A$-module is then an $A$-module  endowed with an  operation of $\fl$ that is $A$-linear. We denote by $U(\fl_A)$ the universal enveloping algebra of the $A$-Lie algebra $\fl_A$.

\begin{definition} Let  $M$ be a
  $\hfg_A$-module. 
\begin{enumerate}
\item We say that $M$ is a {\em weight module} if $M=\bigoplus_{\lambda\in\hfhd} M_\lambda$, where
$$
M_\lambda:=\left\{m\in M\mid H.m=(\lambda(H).1_A+\tau(H))m \text{ for all $H\in \hfh$}\right\}.
$$
We call $M_\lambda$ the {\em weight space} of $M$ corresponding to $\lambda$ (even though its weight is rather $\lambda+\tau$).

\item We say that $M$ is  {\em locally $\hfb_A$-finite}, if for each $m\in M$
  the space  $U(\hfb_A).m$ is a finitely generated $A$-module.
\end{enumerate}
We define $\CO_A$ as the full subcategory of  $\hfg_A\catmod$  that
consists of  locally $\hfb_A$-finite
weight modules.
\end{definition}
One checks easily that $\CO_A$ is an {\em abelian} subcategory of the category of all $\hfg_A$-modules. 
 In the following we write $\CO$ for the non-deformed category, i.e.~ 
for the category $\CO_\DC$ that is defined by giving $\DC$ the
structure of a deformation algebra by identifying it with $S/\fm S$,
where $\fm\subset S$ is the  ideal generated by $\fh\subset
S$. 

 Suppose that $A=\DK $ is a field. Then we can  consider $\hfh_\DK $ and $\hfb_\DK $  as Cartan and Borel subalgebras of $\hfg_\DK $. The $\DC$-linear map $\tau \colon \hfh\to \DK $ induces a $\DK $-linear map $\hfh_\DK \to \DK $ that we denote by $\tau$ as well and which we consider as an element in the dual space $\hfh_\DK ^\star=\Hom_\DK (\hfh_\DK ,\DK )$.  Moreover, each $\lambda\in \hfhd$ induces a $\DK $-linear map $\hfh_\DK \to \DK $, hence  we can consider $\hfhd$ as a subset of $\hfhd_\DK $.
Then $\CO_\DK $ is the full subcategory of the usual category $\CO$ over $\hfg_\DK $ that consists of modules with the property that all weights lie in the set  $\tau+\hfhd\subset \hfhd_\DK $. 

\subsection{The level}
Suppose that $M$ is a weight module. 
Since $\tau(K)=0$, the element $K$ acts on a weight space $M_\lambda$
by multiplication with the scalar $\lambda(K)\in\DC$. For $k\in\DC$ we
denote by $M_k$ the eigenspace of the action of $K$ on $M$ with
eigenvalue $k$. Since $K$ is
central, each eigenspace $M_k$ is a submodule of $M$ and we have $M=\bigoplus_{k\in\DC} M_k$. In the case $M=M_k$ we
call $k$ the {\em level} of the module $M$ and we let
$\CO_{A,k}\subset\CO_{A}$ be the full subcategory whose objects are
those of level $k$. 

It turns out that there is a
distinguished level  $\crit\in\DC$ which is critical in the sense
that the structure of $\CO_{A,\crit}$ differs drastically  from the
structure of $\CO_{A,k}$ for all $k\ne \crit$. For the definition of $\crit$
see Section \ref{sec-crithyp}.

Let $A\to A^\prime$ be a homomorphism of deformation algebras. The following result is easy to prove.
\begin{lemma}\label{lemma-bschg}
 The functor $\cdot\otimes_A A^\prime$ induces a functor $\CO_A\to
 \CO_{A^\prime}$ and for any $k\in\DC$ a functor $\CO_{A,k}\to\CO_{A^\prime,k}$.
\end{lemma}
We denote by $\hfhd_k\subset\hfhd$ the affine hyperplane containing all $\lambda$ with $\lambda(K)=k$.
\subsection{The duality}

For $M\in \CO_A$ we define 
$$
M^{\star}:= \bigoplus_{\lambda\in \hfhd} \Hom_A(M_\lambda,A).
$$
Then $M^\star$ carries an action of $\hfg$ that is given by $(X.\phi)(m)=\phi(-\omega(X).m)$ for  $X\in\hfg$, $\phi\in M^{\star}$
and $m\in M$. Here $\omega\colon \hfg\to\hfg$ is the
Chevalley-involution (cf.~\cite[Section 1.3]{Kac}). It has the
property that it maps the root space $\hfg_\alpha$  to $\hfg_{-\alpha}$ and acts as multiplication by $-1$ on $\hfh$. In particular, we have $(M^\star)_\lambda=\Hom_A(M_\lambda,A)$. Together with the obvious
$A$-module structure, $M^\star$ is an object in
$\CO_A$, and if $M$ is of level $k$, then $M^\star$ is also of
level $k$.

\subsection{The deformed Verma modules}

 For  $\lambda\in\hfhd$ we denote by $A_\lambda$ the $\hfb_A$-module
 that is free of rank one as an $A$-module and on which $\hfb$ acts
 via the character $\lambda+\tau$: this means that  $H\in \hfh$ acts
 as multiplication with the scalar $\lambda(H).1_A+\tau(H)$ and each $X\in [\hfb,\hfb]$ acts by  zero.
The {\em deformed Verma-module} with highest weight $\lambda$ is 
$$
\Delta_A(\lambda):=U(\hfg_A)\otimes_{U(\hfb_A)} A_\lambda.
$$
The {\em deformed dual Verma module} associated to $\lambda$ is
$$
\nabla_A(\lambda):=\Delta_A(\lambda)^\star.
$$

Both $\Delta_A(\lambda)$ and $\nabla_A(\lambda)$  are locally
$\hfb_A$-finite weight modules, hence are contained in $\CO_A$. If
$A\to A^\prime$ is a homomorphism of deformation algebras, then we
have isomorphisms
$$
\Delta_A(\lambda)\otimes_A A^\prime\cong \Delta_{A^\prime}(\lambda), \quad\nabla_A(\lambda)\otimes_A A^\prime\cong \nabla_{A^\prime}(\lambda).
$$

\subsection{Simple objects in $\CO_A$}

Now suppose that $A$ is a local deformation algebra with maximal
ideal $\fm\subset A$ and residue field $\DK =A/\fm$. The residue field
inherits the structure of an $S$-algebra and is, as such, a
deformation algebra as well. The canonical map  $A\to \DK $ gives us a
base change functor $\cdot\otimes_A \DK \colon \CO_A\to \CO_\DK $ by Lemma
\ref{lemma-bschg}. 

As we have observed before, the category $\CO_\DK $ is just a direct
summand of the usual category $\CO$ for the affine Kac--Moody algebra
$\hfg_\DK $. Its objects are those whose weight spaces
correspond to weights in $\tau+\hfhd\subset\hfhd_\DK$. By
the classical theory, the simple isomorphism classes in $\CO_\DK $ are
parametrized by their highest weights in $\tau+\hfhd$, and we denote by
$L_\DK (\lambda)$ a representative corresponding to $\tau+\lambda$.

In \cite[Proposition 2.1]{FieMZ} we showed the following.

\begin{proposition}\label{prop-simples} Suppose that $A$ is a local deformation algebra with residue field $\DK $. Then the functor $\cdot\otimes_A \DK $ yields a bijection
$$
\left\{
\begin{matrix}
\text{simple isomorphism} \\
\text{classes of $\CO_A$}
\end{matrix}
\right\}\stackrel{\sim}\to
\left\{
\begin{matrix}
\text{simple isomorphism} \\
\text{classes of $\CO_\DK $}
\end{matrix}
\right\}.
$$
\end{proposition}

We denote by $L_A(\lambda)$ the simple object corresponding to $L_\DK (\lambda)$ under the above bijection.

\subsection{Characters and Jordan--H\"older multiplicities}\label{sec-JHM} We denote by ``$\varle$'' the usual partial order on $\hfhd$ defined
by $\lambda\varle\mu$ if $\mu-\lambda$ is a sum of positive roots of $\hfg$. Suppose now that $A=\DK $ is a field. In this case we
consider the full subcategory $\CO^f_\DK $ of $\CO_\DK $ that consists of
objects  $M$ such that each weight space $M_\lambda$ is finite
dimensional as a $\DK $-vector space and such that there exist $\mu_1,\dots,\mu_n\in\hfhd$ with the property  that
$M_\lambda\ne 0$ implies $\lambda\le \mu_i$ for some $i$.

Let $\DZ[\hfhd]=\bigoplus_{\lambda\in\hfhd}\DZ e^\lambda$ be the group
ring of the additive group $\hfhd$ and $\widehat{\DZ[\hfhd]}\subset
\prod_{\lambda\in\hfhd}\DZ e^\lambda$ its
completion with respect to the partial order: an element in
$\widehat{\DZ[\hfhd]}$ is an element $\sum_{\lambda\in\hfhd} f_\lambda
e^\lambda$ such that there exist $\mu_1,\dots,\mu_n\in\hfhd$ with the
property that $f_\lambda\ne 0$ implies $\lambda\le \mu_i$ for some
$i$. 
For each $M\in\CO_\DK ^f$ we can then define its character
$$
\cha M:=\sum_{\lambda\in\hfhd}\dim_\DK  M_\lambda\cdot e^\lambda\in
\widehat{\DZ[\hfhd]}.
$$
Now each simple object $L_\DK (\lambda)$ belongs to $\CO_\DK ^f$ and there
are well defined numbers $a_\mu\in \DN$ with
$$
\cha M=\sum_{\mu\in\hfhd} a_\mu \cha L_\DK (\mu).
$$
(cf.~\cite{DGK82}). Note that the sum on the right hand side is in
general an infinite sum. We define the multiplicity of $L_\DK (\mu)$ in
$M$ as
$$
[M:L_\DK (\mu)]:=a_\mu.
$$

\subsection{Truncation}

Our next aim is to study the projective objects in
$\CO_A$. Unfortunately, not all of the $L_A(\lambda)$ admit a
projective cover. In order to overcome this slight technical problem,
we introduce certain truncated subcategories of $\CO_A$ in which a
projective cover exists for each of its simple objects.

Let  $\CJ$ be a subset of $\hfhd$. We call $\CJ$  {\em open} if for
all $\lambda\in \CJ$, $\mu\in\hfhd$ with $\mu\le \lambda$ we have $\mu\in \CJ$. This
indeed defines a topology on $\hfhd$. Note that a subset
$\CI\subset\hfhd$ is closed in this topology if $\lambda\in\CI$,
$\mu\in\hfhd$ with $\mu\ge \lambda$ implies $\mu\in\CI$.

We now construct a functorial filtration on  each object of $\CO_A$ that is indexed by the set of closed  subsets of $\hfhd$ and, dually, a functorial cofiltration indexed by the set of open subsets of $\hfhd$. 

\begin{definition} Suppose that  $\CJ\subset \hfhd$ is  open and let $\CI:=\hfhd\setminus\CJ$ be its closed complement. Let $M\in \CO_A$.
\begin{enumerate}
\item  We define $M_\CI\subset M$ as the $\hfg_A$-submodule generated by the weight spaces corresponding to weights in $\CI$, i.e.~
$$
M_\CI:=U(\hfg_A).\bigoplus_{\lambda\in\CI} M_\lambda.
$$
\item We define
$$
M^\CJ:=M/M_{\CI}.
$$
\end{enumerate}
Let $\CO_{A,\CI}\subset \CO_A$ be the full subcategory of objects $M$ with $M=M_\CI$ and 
 $\CO_{A}^\CJ\subset \CO_A$ the full subcategory of objects $M$ with
 $M=M^\CJ$. 
\end{definition}
Note that an object $M$ of $\CO_A$ belongs to
 $\CO_{A,\CI}$ if and only if it is generated by its weight spaces
 corresponding to weights in $\CI$. Dually, $M$ belongs to $\CO_A^\CJ$
 if and only if $M_\lambda\ne 0$ implies that $\lambda\in\CJ$.

If $\CJ^\prime\subset \CJ$ is another open subset with complement $\CI^\prime\supset \CI$, then we have a natural inclusion $M_{\CI}\subset M_{\CI^\prime}$ and a natural quotient $M^\CJ\to M^{\CJ^\prime}$.  For $\lambda\in \hfhd$, each of the modules $\Delta_A(\lambda)$, $\nabla_A(\lambda)$ and $L_A(\lambda)$ is contained in $\CO_A^{\CJ}$ if and only if $\lambda\in \CJ$. Note that $M\to M_\CI$ defines a functor from $\CO_A$ to $\CO_{A,\CI}$ that is right adjoint to the inclusion $\CO_{A,\CI}\subset \CO_A$. Dually, $M\mapsto M^\CJ$ defines a functor from $\CO_A$ to $\CO_{A}^\CJ$ that is left adjoint to the inclusion $\CO_{A}^\CJ\subset \CO_A$.

\begin{lemma}\label{lemma-truncproj} Suppose that $\CJ$ is an open subset in $\hfhd$ and that  $P$ is a projective object in $\CO_A^\CJ$. Then for any open subset $\CJ^\prime\subset\CJ$, the object $P^{\CJ^\prime}$ is projective in $\CO_A^{\CJ^\prime}$.
\end{lemma}
\begin{proof} This follows immediately from the fact that the functor $(\cdot)^{\CJ^\prime}\colon \CO_A^\CJ\to \CO_A^{\CJ^\prime}$ is left adjoint to the (exact) inclusion functor $\CO_A^{\CJ^\prime}\to\CO_A^\CJ$.
\end{proof}

\begin{lemma}\label{lemma-truncation}  Let $M\in\CO_A$.  Suppose that $\CJ^\prime\subset\CJ\subset\hfhd$ are open subsets. Then there is a canonical isomorphism
  $M^{\CJ^\prime}\stackrel{\sim}\to (M^\CJ)^{\CJ^\prime}$.
\end{lemma}
\begin{proof}
Clearly, the kernel of the quotient $M\to M^{\CJ^\prime}$ as well as the kernel of the composition $M\to M^\CJ\to (M^\CJ)^{\CJ^\prime}$ are generated by the weight spaces $M_\mu$ with $\mu\not\in\CJ^\prime$.
\end{proof}

\subsection{Verma flags}
We start this subsection with a well-known definition.

\begin{definition}\label{def-Vermaflag} Let $M$ be an object in
  $\CO_A$. We say that $M$ {\em admits a Verma flag} if there is a
  finite filtration
$$
0=M_0\subset M_1\subset\dots\subset M_n=M
$$
such that for $i=1,\dots,n$, $M_i/M_{i-1}$ is isomorphic to
$\Delta_A(\mu_i)$ for some $\mu_i\in\hfhd$. 
\end{definition}

Suppose that $M\in\CO_A$ admits a Verma flag. For each $\mu\in\hfhd$,  the number
of occurences of $\Delta_A(\mu)$ as a subquotient of a Verma flag of
$M$ is independent of the chosen filtration. We denote this number by
$(M:\Delta_A(\mu))$.

Let $\mu\in \hfhd$ and $M\in \CO_A$. The set $\CJ=\{\nu\in \hfhd\mid \nu\le \mu\}$ is open and we define $M^{\varle\mu}:= M^{\CJ}$. We define $M^{<\mu}$ likewise. Then we set
$$
M_{[\mu]}:=\ker\left(M^{\varle\mu}\to M^{<\mu}\right).
$$
Note that  $M_{[\mu]}$ is generated by its $\mu$-weight space. If $M$ admits a Verma flag, then $M_{[\mu]}$ is a direct sum of $(M:\Delta_A(\mu))$-copies of $\Delta_A(\mu)$. This follows from the fact that one can reorder each Verma flag such that subquotients corresponding to higher weights occur earlier (see Lemma \ref{lemma-resVermafltrunc} for an analogous result).

\subsection{Projective objects in $\CO_A$}

As before  we  assume that $A$ is a  local deformation algebra with residue field $\DK $.
For general $\lambda$ the simple module $L_A(\lambda)$ admits a
projective cover in $\CO_A$ only if we restrict the set of allowed
weights from above. So let us call a subset $\CJ$ of $\hfhd$ {\em
  bounded} (rather {\em locally bounded from above}) if for any $\lambda\in \CJ$ the set $\CJ\cap\{\ge\lambda\}=\{\mu\in \CJ\mid \mu\ge\lambda\}$ is finite.

\begin{theorem}\label{theorem-projobjinO} Suppose that $A$ is a local deformation algebra with residue field $\DK $.  Let $\CJ$ be a bounded open subset of $\hfhd$. 
\begin{enumerate} 
\item For each $\lambda\in \CJ$ there exists a
  projective cover $P_A^{\CJ}(\lambda)$ of $L_A(\lambda)$ in
  $\CO_A^{\CJ}$. It admits a Verma flag and we have
$$
(P_A^{\CJ}(\lambda):\Delta_A(\mu))=
\begin{cases}
[\nabla_\DK (\mu):L_\DK (\lambda)], & \text{ if $\mu\in \CJ$,} \\
0, & \text{ otherwise.}
\end{cases}
$$
\item If $\CJ^\prime\subset \CJ$ is open as well, then 
$$
P_A^{\CJ}(\lambda)^{\CJ^\prime}\cong P_A^{\CJ^\prime}(\lambda).
$$
\item If $A\to A^\prime$ is a homomorphism of local deformation
  algebras and $P\in\CO_A^\CJ$ is projective, then $P\otimes_A
  A^\prime\in\CO_{A^\prime}^\CJ$ is projective. 
\item We have $P_A^{\CJ}(\lambda)\otimes_A \DK \cong P_\DK ^{\CJ}(\lambda)$. 
\item Suppose that  $P$ is  a finitely generated projective object in
  $\CO_A^\CJ$ and that $A\to A^\prime$ is a homomorphism of local deformation
  algebras. For any $M\in\CO_A^\CJ$ the natural map 
$$
\Hom_{\CO_A}(P,M)\otimes_A A^\prime\to \Hom_{\CO_{A^\prime}}(P\otimes_A A^\prime,M\otimes_A A^\prime)
$$
is an isomorphism.
\end{enumerate}
\end{theorem}

\begin{proof} Part (1) is contained in \cite[Theorem 4.2, Theorem 5.3]{Fie11}. Part (2) is shown in the course of the proof of Theorem 4.2 in \cite{Fie11}. The statements in (3) and (5) are found in \cite[Proposition 2.4]{FieMZ}. Finally, part (4) is shown in the course of the proof of Theorem 5.3 in \cite{Fie11}.
\end{proof}

\subsection{The block decomposition of $\CO_A$}

Let $A$ be a local deformation algebra with residue field $\DK $.
We let $\sim_A$ be the equivalence relation on $\hfhd$ that is generated
by the relations $\lambda\sim_A\mu$ for all $\lambda,\mu\in\hfhd$ for which
there exists an open bounded subset $\CJ$ of $\hfhd$ such that $L_A(\mu)$ is a subquotient of $P_A^\CJ(\lambda)$, i.e. if there is a non-zero homomorphism $P_A^\CJ(\mu)\to P_A^\CJ(\lambda)$. 

\begin{lemma}\label{lemma-eqrel} The equivalence relation $\sim_A$ is also generated by
  either of the following sets of relations:
\begin{enumerate}
\item $\lambda\sim_A\mu$ if there exists an open bounded
subset $\CJ$ of $\hfhd$ such that $(P_A^{\CJ}(\lambda):\Delta_A(\mu))\ne
0$.
\item $\lambda\sim_A\mu$ if $[\Delta_\DK (\lambda):L_\DK (\mu)]\ne 0$. 
\end{enumerate}
\end{lemma}
\begin{proof} See \cite[Lemma 5.5]{Fie11}.
\end{proof}

For an equivalence class $\Lambda\in\hfhd/\sim_A$ we define the full
subcategory $\CO_{A,\Lambda}$ of $\CO_A$ that contains all objects $M$
that have the property that each highest weight of a subquotient lies
in $\Lambda$. Note that it is the subcategory generated by the objects
$P_A^{\CJ}(\lambda)$ for all $\lambda\in\Lambda$ and all bounded open
subsets $\CJ$ of $\hfhd$ that contain $\lambda$. 
Then we have the following result on the decomposition of $\CO_A$.
\begin{theorem}[{\cite[Theorem 5.1]{Fie11}}] The functor
\begin{align*}
\prod_{\Lambda\in\hfhd/{\scriptstyle \sim_A}} \CO_{A,\Lambda} &\to \CO_A\\
(M_\Lambda)&\mapsto \bigoplus_\Lambda M_\Lambda
\end{align*}
is an equivalence of categories.
\end{theorem}

\subsection{The Kac--Kazhdan theorem, integral roots and the integral Weyl group}

The Kac--Kazhdan theorem gives a rather explicit description of the set
of pairs $(\lambda,\mu)$ such that $[\Delta_\DK (\lambda):L_\DK (\mu)]\ne
0$. By Lemma \ref{lemma-eqrel}, these pairs generate the equivalence relation
``$\sim_A$''. 

Recall the bilinear form $(\cdot,\cdot)\colon\hfhd\times\hfhd\to
\DC$. For any deformation algebra $A$ we set $\hfhd_A:=\hfhd\otimes_\DC A=\Hom_\DC(\hfh,A)$ and denote
by $(\cdot,\cdot)_A\colon\hfhd_A\times\hfhd_A\to A$ the $A$-bilinear
continuation of $(\cdot,\cdot)$. The structure map $\tau\colon \hfh\to
A$ can be considered as an element in $\hfhd_A$. Let $\rho\in\hfhd$ be
an element with $(\rho,\alpha)=1$ for any simple affine root
$\alpha\in\widehat\Pi$.

Now we can state the result of Kac and Kazhdan (we slightly reformulate
their original theorem in terms of equivalence classes): 

\begin{theorem}[\cite{KK79}]
 The relation ``$\sim_A$'' is generated by $\lambda\sim_A\mu$
  for all pairs $\lambda,\mu$ such that there exists a root $\alpha\in
  \hR$ and $n\in\DZ$ with $2(\lambda+\rho,\alpha)_\DK =n(\alpha,\alpha)_\DK $
  and $\lambda-\mu=n\alpha$. 
\end{theorem}

For $\lambda\in\hfhd$ we define the set of {\em integral roots}
(with respect to $\lambda$) by
\begin{align*}
\hR_A(\lambda)&:=\{\alpha\in\hR\mid
2(\lambda+\rho,\alpha)_\DK \in\DZ(\alpha,\alpha)_\DK \}
\end{align*}
and the corresponding {\em integral Weyl group} by
$$
\hCW_A(\lambda):=\langle s_\alpha\mid
\alpha\in\hR_A(\lambda)\cap\hR^{\re}\rangle\subset\hCW.
$$
Let $\Lambda\subset\hfhd$ be an equivalence class with respect to
``$\sim_A$''. 
It follows
from the Kac--Kazhdan theorem that we have $\hR_A(\lambda)=\hR_A(\mu)$ and
$\hCW_A(\lambda)=\hCW_A(\mu)$ for all $\lambda,\mu\in\Lambda$. Hence we
can denote these two objects by $\hR_A(\Lambda)$ and
$\hCW_A(\Lambda)$.

\subsection{The critical level}\label{sec-crithyp}

Let $\Lambda\in\hfhd/\sim_A$ be an equivalence class. For each $\lambda,\mu\in\Lambda$ we then have $\lambda(K)=\mu(K)$, hence there is a certain
$k=k(\Lambda)\in\DC$ such that each object in $\CO_{A,\Lambda}$ is of
level $k$. Note that $\nu(K)=(\nu,\delta)$ for all $\nu\in\hfhd$. 

\begin{lemma}[{\cite[Lemma 4.2]{AF08}}]\label{lemma-critlevel} Let $\Lambda\in\hfhd/\sim_A$ be an equivalence class. The following are equivalent.
\begin{enumerate}
\item We have $\lambda(K)=-\rho(K)$ for some $\lambda\in\Lambda$.
\item We have  $\lambda(K)=-\rho(K)$ for all $\lambda\in\Lambda$.
\item We have $\lambda+\delta\in\Lambda$ for all $\lambda\in\Lambda$.
\item We have $n\delta\in\hR_A(\Lambda)$ for some $n\ne 0$. 
\item We have $n\delta\in\hR_A(\Lambda)$ for all $n\ne 0$. 
\end{enumerate}
\end{lemma}
The level $\crit:=-\rho(K)$ is called the {\em critical} level.

\section{Restricted representations}

In this section we
recall one of the most significant structures that we encounter for the
category $\CO$ of an affine Kac--Moody algebra at the critical level. Recall that we add the
derivation operator $D$ to the central extension of the loop algebra
corresponding to $\fg$. This allows us to consider $\CO$ (and the
deformed versions $\CO_A$) as graded categories.

\subsection{A shift functor}

Suppose that $M$ is a $\hfg_A$-module and that $L$ is a
$\hfg=\hfg_\DC$-module. Then $M\otimes_\DC L$ acquires the structure
of a $\hfg_A$-module such that
$\hfg$ acts  via the usual tensor product action ($X(m\otimes
l)=Xm\otimes l+m\otimes Xl$ for $X\in\hfg$, $m\in M$, $l\in L$) and $A$ acts on  the first tensor factor. The following is easy to prove.

\begin{lemma} 
\begin{enumerate}
\item If $M$ is locally $\hfb_A$-finite and $L$ is locally $\hfb$-finite, then $M\otimes_\DC L$ is locally $\hfb_A$-finite.
\item If $M$ and $L$ are weight modules, then $M\otimes_\DC L$ is a weight module.
\item If $M\in \CO_A$ and $L\in \CO$, then $M\otimes_\DC L\in \CO_A$.
\end{enumerate}
\end{lemma}

Note that the simple module $L(\delta)=L_\DC(\delta)$ is one-dimensional (the subalgebra $\tfg=[\hfg,\hfg]$ acts trivially, while $D$ acts as the identity operator). The $\hfg$-module $L(\delta)\otimes_\DC L(-\delta)\cong L(0)$ is the trivial module.  
In particular, the {\em shift functor} 
\begin{align*}
T\colon \CO_A &\to \CO_A\\
M&\mapsto M\otimes_\DC L(\delta)
\end{align*}
is an equivalence with inverse $T^{-1}=\cdot\otimes_\DC L(-\delta)$. 
Since $L(\delta)$ has level
$0$ the shift functor $T$ preserves the subcategories $\CO_{A,k}$, i.e.~ we get induced autoequivalences $T\colon \CO_{A,k}\to \CO_{A,k}$ for each $k$. 

Let $\Lambda\in\hfhd/_{\textstyle \sim_A}$ be an equivalence class. The corresponding block $\CO_{A,\Lambda}$ is preserved by the functor $T$ if and only if for each $\lambda\in\Lambda$ we have $\lambda+\delta\in\Lambda$, hence if and only if $\Lambda$ is critical (cf.~Lemma \ref{lemma-critlevel}). 

In the following we will study natural transformations $z\colon T^n \to \id$ (for some $n\in\DZ$) from the functor $T^n$ to the identity functor (on $\CO_A$, $\CO_{A,k}$ or $\CO_{A,\Lambda}$). Note that, if $k\ne \crit$,  then there is no non-vanishing natural transformation from $T^n$ to $\id_{\CO_{A,k}}$ if $n\ne 0$.  In contrast, for $k=\crit$ the space of natural transformations from $T^n$ to $\id_{\CO_{A,\crit}}$ is huge.

\subsection{The Feigin--Frenkel center}\label{sec-FFcenter}
\newcommand{\FFC}{\mathfrak{z}} For a more thorough discussion of the structure that we introduce now we refer to Section 5 of \cite{AF08}. 
We denote by $V^\crit(\fg)$ the universal affine vertex algebra associated with $\fg$ at the critical level and by $\FFC$ its center. Then each smooth $\tilde\fg=[\hfg,\hfg]$-module $M$ can be considered as a graded module over the vertex algebra $V^\crit(\fg)$ and hence over $\FFC$. In \cite{AF08} we exhibited homogeneous generators $p^{(1)},\dots,p^{(l)}$, where $l$ denotes the rank of $\fg$, of $\FFC$ and this yields an action of the graded polynomial ring 
$$
\CZ_\crit=\DC[p_s^{(i)}\mid i=1,\dots,l, s\in \DZ]=\bigoplus_{n\in\DZ}\CZ_\crit^n
$$
on $M$. Here, $\CZ_\crit^n$ is the subspace of $\CZ_\crit$ spanned by the elements $p_{n_1}^{(i_1)}\cdots p_{n_r}^{(i_r)}$ with $n_1+\cdots+n_r=n$. We set $\CZ_\crit^-=\bigoplus_{n<0}\CZ_\crit^n$, $\CZ^+_\crit=\bigoplus_{n>0}\CZ_\crit^n$, $\CZ^{\ge 0}_\crit=\bigoplus_{n\ge0}\CZ^n_\crit$ and $\CZ^{\le 0}_\crit=\bigoplus_{n\le 0}\CZ^n_\crit$.

Now $\hfg=\tilde\fg\oplus\DC D$ and the action of the grading operator $D$ allows us to view  each $z\in\CZ^n_\crit$ as a natural transformation from $T^n$ to the identity functor on $\CO_{A,\crit}$. For $M\in\CO_{A,\crit}$ we denote by $z^M\colon T^nM\to M$ the resulting homomorphism.
This natural transformation is compatible with the base change functors $\CO_{A,\crit}\to \CO_{A^\prime,\crit}$ associated to a homomorphism $A\to A^\prime$ of deformation algebras, in the sense that $z^{M\otimes_AA^\prime}=z^M\otimes \id\colon T^n(M\otimes_A A^\prime)=(T^nM)\otimes_AA^\prime\to M\otimes_A A^\prime$.

\subsection{Restricted representations}

Let $A$ be a local deformation algebra. 
\begin{definition}\label{def-resrep} Let $M\in\CO_{A,\crit}$. We say that $M$ is {\em
    restricted} if for all $n\ne 0$ and all $z\in\CZ_{\crit}^n$ the homomorphism  $z^M\colon T^nM\to M$ is zero.
  \end{definition}
 We denote by $\rCO_{A,\crit}$ the full subcategory of $\CO_{A,\crit}$ that
 consists of restricted representations. For an open subset $\CJ$ of
 $\hfhd$ we set $\rCO_{A,\crit}^\CJ=\rCO_{A,\crit}\cap\CO_{A,\crit}^{\CJ}$.

For $M\in\CO_{A,\crit}$ we define $M_\res$, the largest restricted submodule, and $M^{\res}$, the largest restricted quotient, as follows. For each $z\in\CZ_{\crit}^n$ we can view $z^{T^{-n}M}$ as a homomorphism from $T^{-n}T^nM=M$ to $T^{-n}M$. Then 
$$
M_{\res} = \{ m\in M \mid z^{T^{-n}M}(m)=0\text{ for all $z\in\CZ_{\crit}^n$, $n\ne 0$}\}.
$$
Let $\CZ_\crit^nM$ be the submodule of $M$ generated by the images of all homomorphisms $z^M\colon T^n M\to M$ with $z\in\CZ_\crit^n$. Then
$$
M^{\res} = M/ \sum_{n\in \DZ,n\ne 0} \CZ_{\crit}^n  M.
$$
Both $M_{\res}$ and $M^{\res}$ are restricted objects in $\CO_{A,\crit}$.    We get functors $M\mapsto M_{\res}$ and $M\mapsto M^{\res}$ from $\CO_{A,\crit}$ to $\rCO_{A,\crit}$ that are right  resp. left adjoint to the inclusion functor $\rCO_{A,\crit}\to \CO_{A,\crit}$.

\subsection{Restriction, truncation and base change}
We now collect some  results on the restriction functor.
\begin{lemma}\label{lemma-resandtrunc} Let $\CJ\subset\hfhd$ be open. For each $M\in\CO_A$
  there is a natural isomorphism
$$
(M^{\res})^\CJ\cong (M^\CJ)^{\res}.
$$
\end{lemma}
\begin{proof} The kernel of both compositions  $M\to M^{\res}\to (M^{\res})^\CJ$ and $M\to M^\CJ\to (M^\CJ)^{\res}$ is generated by all weightspaces $M_\mu$ with $\mu\not\in\CJ$ together with $\sum_{n\ne 0}\CZ_\crit^n M$. 
\end{proof}

\begin{lemma}\label{lemma-resbschg} Let $M\in\CO_A$ and fix a homomorphism $A\to A^\prime$
  of deformation algebras. Then there is a canonical isomorphism 
$$
(M\otimes_A A^\prime)^{\res}\stackrel{\sim}\to (M^{\res}\otimes_A A^\prime)^\res.
$$
\end{lemma}

\begin{proof} We consider the canonical homomorphisms $a\colon M\otimes_A
  A^\prime\to (M\otimes_A A^\prime)^{\res}$ and $b\colon M\otimes_A
  A^\prime\to M^{\res}\otimes_A A^\prime\to (M^{\res}\otimes_A
  A^\prime)^{\res}$ and we show that $\ker a=\ker b$. Note that the
  kernel of $a$ is generated by the subspaces
  $\CZ_{\crit}^n(M\otimes_A A^\prime)=(\CZ_{\crit}^n M)\otimes_AA^\prime$ for $n\ne 0$, and the kernel
  of $b$ is generated by the spaces $\CZ_{\crit}^n(M\otimes_A A^\prime)$, $n\ne 0$, and $(\CZ_A^m M)\otimes_A A^\prime$,  $m\ne 0$, so clearly $\ker a=\ker b$. 
 \end{proof}

\subsection{Restricted Verma modules}\label{sec-resVerma}

For each critical $\lambda\in\hfhd$ we define the restricted Verma module  by 
$$
\rDelta_A(\lambda) :=\Delta_A(\lambda)^{\res}
$$
and the restricted dual Verma module by 
$$
\rnabla_A(\lambda) := \nabla_A(\lambda)_{\res}.
$$

We clearly have $\CZ_\crit^+\Delta_A(\lambda)=0$ and $\CZ_\crit^-\nabla_A(\lambda)=0$. Hence we obtain

\begin{lemma} For each critical $\lambda\in\hfhd$ we have 
$\rDelta_A(\lambda)=\Delta_A(\lambda)/\CZ^-_\crit\Delta_A(\lambda)$, and 
$\rnabla_A(\lambda)\subset\nabla_A(\lambda)$ is the set of $\CZ^+_\crit$-invariant elements. 
\end{lemma}

\subsection{The character of a restricted Verma module}

Let us define the numbers $p(n)\in \DN$ for $n\ge 0$ by the following equation (in $\widehat{\DZ[\hfhd]}$)
$$
\prod_{l\ge 0} (1+e^{-l\delta}+e^{-2l\delta}+\dots)^{\rk\, \fg}=\sum_{n\ge 0} p(n)e^{-n\delta},
$$
and the numbers $q(n)\in\DZ$, $n\ge 0$ by the corresponding equation for the inverse of the left hand side:
$$
\left(\prod_{l\ge 0} (1+e^{-l\delta}+e^{-2l\delta}+\dots)^{\rk\, \fg}\right)^{-1}=\prod_{l\ge 0} (1-e^{-l\delta})^{\rk\, \fg}=\sum_{n\ge 0} q(n)e^{-n\delta}.
$$

\begin{lemma}\label{lemma-resmul} Suppose that $A=\DK $ is a field. Let $\lambda\in \hfhd$ be critical.
\begin{enumerate}
\item We have
$$
\cha \rDelta_\DK (\lambda)=e^\lambda\prod_{\alpha\in \hR^{+,\re}} (1+e^{-\alpha}+e^{-2\alpha}+\dots).
$$
\item For all $\mu\in\hfhd$ we have 
$$
[\rDelta_\DK (\lambda):L_\DK (\mu)]=\sum_{n\ge 0}q(n)[\Delta_\DK (\lambda-n\delta):L_\DK (\mu)].$$
\end{enumerate}
\end{lemma}
\begin{proof} The first statement is due to  Feigin--Frenkel and Frenkel (cf.~ the proof of Theorem 6.4.1 in \cite{Fre07}). Using the well-known character formula for 
the usual Verma modules we get 
\begin{align*}
\cha\Delta_\DK (\lambda)&=e^\lambda\prod_{\alpha\in \hR^{+}} (1+e^{-\alpha}+e^{-2\alpha}+\dots)^{\dim\hfg_{\alpha}}\\
&=\prod_{l> 0} (1+e^{-l\delta}+e^{-2l\delta}+\dots)^{\rk\, \fg}\cha\rDelta_\DK (\lambda).
\end{align*}
(Note that $\dim \hfg_{\alpha}=1$ for real roots $\alpha$, and $\dim \hfg_{l\delta}=\rk\,\fg$ for all $l\ne 0$.
Dividing this equation by $\prod_{l> 0} (1+e^{-l\delta}+e^{-2l\delta}+\dots)^{\rk\, \fg}$ yields
\begin{align*}
\cha\rDelta_\DK (\lambda)&=\left(\prod_{l> 0} (1+e^{-l\delta}+e^{-2l\delta}+\dots)^{\rk\, \fg}\right)^{-1}\cha\Delta_\DK (\lambda)\\
&=\sum_{n\ge0} q(n)e^{-n\delta}\cha\Delta_\DK (\lambda) \\
&=\sum_{n\ge0} q(n)\cha\Delta_\DK (\lambda-n\delta),
\end{align*}
hence (2).
\end{proof}

\subsection{Restricted Verma modules over local rings}

The following is an easy consequence of Nakayama's lemma:

\begin{lemma}\label{lemma-freeAmod} Let $A$ be a local domain with residue field $\DK$ and quotient field $Q$. Let $M$  be a finitely generated $A$-module and suppose that
$$
\dim_\DK  M\otimes_A \DK =\dim_Q M\otimes_A Q.
$$
Then $M$ is a free $A$-module with  $\rk_A M=\dim_\DK  M\otimes_A \DK =\dim_Q M\otimes_A Q$. 
\end{lemma}

From now on  let $A$ be a local deformation domain with residue field
  $\DK $ and quotient field $Q$.

\begin{lemma}\label{lemma-resVermafree} Suppose $\lambda\in\hfhd$ is critical. Then the following holds. For any  $\mu\in\hfhd$ the weight space
  $\rDelta_A(\lambda)_\mu$ is a free $A$-module of rank
$$
\rk_A \rDelta_A(\lambda)_\mu=\dim_\DK \rDelta_\DK (\lambda)_\mu.
$$
\end{lemma}

\begin{proof} 
The base change remark in Section \ref{sec-FFcenter} shows that we have isomorphisms
 $$
 \rDelta_A(\lambda)\otimes_A Q\cong\rDelta_Q(\lambda), \quad  \rDelta_A(\lambda)\otimes_A \DK \cong\rDelta_\DK (\lambda).
 $$
 As these isomorphisms induce isomorphisms on any weight space and since the weight space dimensions coincide by Lemma \ref{lemma-resmul}, we can apply Lemma \ref{lemma-freeAmod}, which immediately yields the statement that we want to prove.
\end{proof}

\begin{lemma} Let  $\lambda\in\hfhd$ be critical. 
Then we have $\rDelta_A(\lambda)^\star\cong \rnabla_A(\lambda)$, $\rnabla_A(\lambda)^\star\cong \rDelta_A(\lambda)$.
\end{lemma}
\begin{proof} Note that by Lemma \ref{lemma-resVermafree}, each weight space of $\rDelta_A(\lambda)$ is a free $A$-module of finite rank, so it is reflexive, i.e.~ $(\rDelta_A(\lambda)^{\star})^{\star}=\rDelta_A(\lambda)$. Hence it is enough to prove that $\rDelta_A(\lambda)^{\star}\cong\rnabla_A(\lambda)$.

We consider now the short exact sequence
$$
0\to \sum_{n\ne 0} \CZ_\crit^n  \Delta_A(\lambda)\to \Delta_A(\lambda)\to \rDelta_A(\lambda)\to 0.
$$
As each weight space of $\Delta_A(\lambda)$ and of $\rDelta_A(\lambda)$
is a free $A$-module of finite rank, the sequence above splits as a
sequence of $A$-modules. Hence each weight space of 
$\sum_{n\ne 0} \CZ_\crit^n  \Delta_A(\lambda)$ is free and the dual sequence
$$
0\to \rDelta_A(\lambda)^{\star}\to\nabla_A(\lambda)\to\left(\sum_{n\ne 0} \CZ_\crit^n  \Delta_A(\lambda)\right)^\star\to  0
$$
is exact as well.

The injective map factors over the inclusion $\rnabla_A(\lambda)\to \nabla_A(\lambda)$, as $\rDelta_A(\lambda)^{\star}$ is restricted. By definition, the composition of $\rnabla_A(\lambda)\to \nabla_A(\lambda)$ 
with the surjection $\nabla_A(\lambda)\to (\sum_{n\ne 0} \CZ_\crit^n  \Delta_A(\lambda))^\star$ is zero. Hence 
$ \rDelta_A(\lambda)^{\star}\cong\rnabla_A(\lambda)$.
\end{proof}

\subsection{An auxiliary category}
In the following it is convenient to work with "half-restricted" objects. 
 \begin{definition} We let $\CO_{A,{\crit}}^-$ be the full subcategory of $\CO_{A,{\crit}}$ that consists of all objects $M$ such that  $\CZ^-_{\crit} M=0$. For an open bounded subset $\CJ$ of $\hfhd$ we let $\CO_{A,\crit}^{-\CJ}$ be the full category of $\CO_{A,\crit}^-$ of objects that are also contained in $\CO_{A,\crit}^\CJ$.
 \end{definition}

It is clear that $\CO^-_{A,\crit}$ and $\CO_{A,\crit}^{-\CJ}$ are stable under taking quotients or subobjects. For a base change homomorphism $A\to A^\prime$ and an object $M$ of $\CO_{A,\crit}^-$ we have that $M\otimes_A A^\prime$ is contained in $\CO_{A^\prime,\crit}^-$.  Note that a critical Verma module does not belong to $\CO^-_{A,\crit}$, but each critical restricted Verma module does. Also the dual non-restricted critical Verma modules belong to $\CO^-_{A,\crit}$.

In analogy to the restriction functors $M\mapsto M^{\res}$, $M\mapsto M_{\res}$ we have functors $M\mapsto M^-$, $M\mapsto M_-$ that are left resp.~ right adjoint to the inclusion of $\CO_{A,\crit}^-$ in $\CO_{A,\crit}$. For example, $M^-$ is the quotient of $M$ by the submodule $\CZ_{\crit}^- M$.

\subsection{Restricted Verma flags}

Now we state the definition of a restricted Verma flag in analogy to
Definition \ref{def-Vermaflag}.

\begin{definition} We say that a module $M\in \CO^-_{A,{\crit}}$ {\em admits a
    restricted Verma flag} if there is a finite filtration
$$
0=M_0\subset M_1\subset\dots\subset M_n=M
$$
such that for each $i=1,\dots,n$, $M_i/M_{i-1}$ is isomorphic to
$\rDelta_A(\mu_i)$ for some $\mu_i\in\hfhd$. 
\end{definition}
Again, if $M\in\CO^-_{A,{\crit}}$ admits a restricted Verma flag, then for each $\mu\in\hfhd$
the number of occurences of $\rDelta_A(\mu)$ is independent of the
chosen filtration. We denote this number by $(M:\rDelta_A(\mu))$. 

Note that if $M\in\CO_{A,\crit}^-$ admits a restricted Verma flag, then so does $M\otimes_A A^\prime\in\CO_{A^\prime,\crit}^-$ for any homomorphism $A\to A^\prime$ of deformation algebras (as a restricted Verma module is free over the deformation algebra and $\rDelta_A(\lambda)\otimes_A A^\prime\cong\rDelta_{A^\prime}(\lambda)$), and we have
$$
(M\otimes_A A^\prime:\rDelta_{A^\prime}(\mu))=(M:\rDelta_A(\mu)).
$$

\begin{proposition}\label{Pro:exact-on-Delta} \begin{enumerate}
\item Suppose that $M\in\CO_{A,\crit}$ admits a Verma flag. Then $M^-$ admits a restricted Verma flag and we have
$
(M^-:\rDelta_A(\lambda))=(M:\Delta_A(\lambda))
$
for any $\lambda\in\hfhd$.
\item
Let 
$0\to M\to N\to L\to 0$ be an exact sequence in $\CO_{A,\crit}$ and suppose that $M$, $N$ and $O$ admit a Verma flag.
Then
the induced sequence
\begin{align*}
0\to M^-\to  N^-\to  L^-\to 0
\end{align*}
is also exact.
\end{enumerate}
\end{proposition}
\begin{proof} This follows easily from the fact that each Verma module is free over $\CZ_{\crit}^{\le 0}$ (cf.~\cite[Theorem 9.5.3]{Fre07}).   
\end{proof}

\begin{lemma}\label{lemma-projdomresVerma} Let $M\in\CO^-_{A,{\crit}}$ and
let  $\lambda\in\hfhd$ be maximal with $M_\lambda\ne 0$. Then each
  surjective map $M\to\rDelta_A(\lambda)$ splits. 
\end{lemma}
\begin{proof} Let $x\in M_\lambda$ be a preimage of a generator of
  $\rDelta_A(\lambda)$. By maximality of $\lambda$ there is a
  homomorphism $\Delta_A(\lambda)\to M$ that sends a generator of
  $\Delta_A(\lambda)$ to $x$. As $M$ is in $\CO_{A,{\crit}}^-$, this map factors over the quotient map $\Delta_A(\lambda)\to\Delta_A(\lambda)^-=\rDelta_A(\lambda)$. 
\end{proof}

\begin{lemma}\label{lemma-resVermafltrunc} Let $M\in\CO^-_{A,{\crit}}$.
\begin{enumerate}
\item Suppose that $M$ admits a restricted Verma flag and let
  $\{\mu_1,\dots,\mu_l\}$ be an enumeration of the multiset that
  contains each $\mu\in\hfhd$ with multiplicity
  $(M:\rDelta_A(\mu))$. Suppose furthermore that this enumeration has the property that $\mu_i>\mu_j$ implies $i<j$. Then there is a filtration $0=M_0\subset
  M_1\subset\dots\subset M_l=M$ with
  $M_i/M_{i-1}\cong\rDelta_A(\mu_i)$ for each $i=1,\dots, l$. 
  \end{enumerate}
  Let $\CJ$ be an open subset of
  $\hfhd$ and let $\CI:=\hfhd\setminus\CJ$ be its complement.
  \begin{enumerate}\setcounter{enumi}{1} 
\item  $M$ admits a restricted
  Verma flag if and only if both $M_\CI$ and $M^\CJ$ admit restricted
  Verma flags.
\item If $M$ admits a restricted Verma flag, 
  then we have for all $\mu\in\hfhd$
\begin{align*}
(M_\CI:\rDelta_A(\mu))&=
\begin{cases}
(M:\rDelta_A(\mu)),&\text{ if $\mu\in\CI$},\\
0,&\text{ otherwise},
\end{cases}\\
(M^\CJ:\rDelta_A(\mu))&=
\begin{cases}
(M:\rDelta_A(\mu)),&\text{ if $\mu\in\CJ$},\\
0,&\text{ otherwise}.
\end{cases}
\end{align*}
\end{enumerate}
\end{lemma}

\begin{proof} Part (1) follows directly from Lemma
  \ref{lemma-projdomresVerma}. So let us prove (2) and (3). Consider the short exact sequence $0\to M_\CI\to M\to
  M^\CJ\to 0$. Clearly, if
  $M_\CI$ and $M^\CJ$ admit restricted Verma flags, then so does
  $M$. So suppose that $M$ admits restricted Verma flag. By (1) we can
  find a filtration $0=M_0\subset M_1\subset\dots\subset M_l=M$  such
  that $M_i/M_{i-1}\cong\rDelta_A(\mu_i)$ and such that
  $\{\mu_1,\dots,\mu_n\}\subset\CI$ and
  $\{\mu_{n+1},\dots,\mu_l\}\subset\CJ$ for some $n\ge 0$. We then
  have $M_\CI=M_n$, as $M_n$ is generated by its vectors of weights
  $\mu_1,\dots,\mu_n$ and the weights of  $M/M_n$ belong to
  $\CJ$. Hence $M^\CJ=M/M_n$ and we deduce that both $M_\CI$ and
  $M^\CJ$ admit a restricted Verma flag and that the multiplicity
  statements in (3) hold as well.  
\end{proof}

\section{Restricted projective objects}\label{sec-Projectives}

Let $A$ be a local deformation algebra and $\CJ\subset\hfhd$ a  bounded open subset. For each  $\lambda\in\CJ$ we have a projective cover $P_A^\CJ(\lambda)\to L_A(\lambda)$ in $\CO_{A,\crit}^\CJ$. By applying the functor $(\cdot)^{\res}$ we obtain a surjective map $P_A^\CJ(\lambda)^\res\to L_A(\lambda)^\res=L_A(\lambda)$. As $(\cdot)^\res$ is left adjoint to the (exact) inclusion functor $\rCO_{A,\crit}^{\CJ}\subset\CO_{A,\crit}^\CJ$, $P_A^\CJ(\lambda)^\res$ is a projective object in $\rCO_{A,\crit}^{\CJ}$. It is even indecomposable as $P_A^\CJ(\lambda)$ has a unique simple quotient (see Section 4.3 in  \cite{Fie11}). We define
$$
\rP_A^\CJ(\lambda):=P_A^\CJ(\lambda)^{\res}.
$$

Similarly, using the functor $(\cdot)^-$ instead of $(\cdot)^\res$, we can define a restricted projective cover $P_A^\CJ(\lambda)^-\to L_A(\lambda)$ in the category $\CO_{A,\crit}^{-\CJ}$. In particular, $\CO_{A,\crit}^{-\CJ}$ contains enough projectives, so we can calculate $\Ext$-groups.

By Proposition \ref{Pro:exact-on-Delta},  $P_A^\CJ(\lambda)^-$ admits a restricted Verma flag with multiplicities
$$
(P_A^\CJ(\lambda)^-:\rDelta_A(\mu))=(P_A^\CJ(\lambda):\Delta_A(\mu)).
$$
One of the main results in this article is that $\rP_A^\CJ(\lambda)$ admits restricted Verma flag as well and that the multiplicities are given by a BGG-reciprocity formula.

\subsection{An $\Ext$-vanishing criterion}
We now prove a result that is well-known in similar, more classical situations.
For this we need to assume that $A=\DK$ is a field. Let $\CJ\subset\hfhd$ be open and bounded.

\begin{proposition}\label{prp:Ext-vs-Verma-flag-1}
Let $X\in \CO_{\DK,\crit}^{-\CJ}$.
The following conditions
are equivalent:
\begin{enumerate}
\item $X$ admits a restricted Verma flag. 
\item $X$ is finitely generated 
and $\Ext_{\CO_{\DK,\crit}^{-\CJ}}^i(X,\nabla_\DK(\lambda))=0$ for any $i\geq 1$ and any $\lambda\in \CJ$.
\item $X$ is finitely generated 
and
$\Ext_{\CO_{\DK,\crit}^{-\CJ}}^1(X, \nabla_\DK(\lambda) )=0$
for any  $\lambda\in \CJ$.
\end{enumerate}
\end{proposition}
\begin{proof} For brevity we write $\Ext^i$ for $\Ext^i_{\CO_{\DK,\crit}^{-\CJ}}$ in the course of this proof.
We show that (1) implies  (2). So suppose that $X$ admits a restricted Verma flag. 
It is clear that each module admitting a restricted Verma flag is finitely generated. We prove
the vanishing of $\Ext^i$ by induction on $i$.

Let $i=1$ and
let
$X=X_0\supset X_1\supset X_2\supset \dots\supset X_l=0$
be a restricted Verma flag of $X$.
From the exact sequence
$$0\rightarrow X_1\rightarrow X\rightarrow X/X_1\rightarrow
0,$$
we obtain
the exact sequence
$$\Ext^1(X/X_1, \nabla_\DK(\lambda) )
\rightarrow
\Ext^1(X,  \nabla_\DK(\lambda) )
\rightarrow
\Ext^1(X_1, \nabla_\DK(\lambda)  ).
$$
Now $X_1$ and $X/X_1$ admit a restricted Verma flag. 
Hence, by using 
induction on $l$,
one sees that
it is enough
to prove the case when $l=1$.
So let
$X= \rDelta_\DK(\mu)$ with $\mu\in \CJ$,
and let
\begin{align*}
 0\to  \nabla_\DK(\lambda) \to M\to \rDelta_\DK(\mu)\to 0
\end{align*}
be an exact sequence in $\CO_{\DK,\crit}^{-\CJ}$. We have to show that this sequence splits.

If $\mu\geq \lambda$,
then this splits by Lemma \ref{lemma-projdomresVerma}.
If $\mu\not\geq \lambda$, apply the duality functor 
and consider the exact sequence
\begin{align*}
 0\to \rnabla_\DK(\mu)\to M^\star \to \Delta_\DK(\lambda)\to 0
\end{align*}
in the category $\CO_{\DK,\crit}^{\CJ}$ (note that this is not a sequence in $\CO_{\DK,\crit}^{-\CJ}$!).
This splits
because $\lambda$ is a maximal weight in $M^\star$, hence the former sequence splits as well, so $\Ext^1(\rDelta_\DK(\lambda),\nabla_\DK(\mu))=0$. We have now proved statement (2) of the proposition for $i=1$.

Next let $i\geq 2$.
By the same argument as above,
it is enough to consider the case
$X=\rDelta_\DK(\mu)$ for some  $\mu\in\CJ$. One has an exact sequence
\begin{align*}
 0\to N\to P_\DK^\CJ(\mu)\to \Delta_\DK(\mu)\to 0
\end{align*}
in the category $\CO^\CJ_{\DK,\crit}$ and each occuring module admits a (non-restricted) Verma flag.
By Proposition \ref{Pro:exact-on-Delta},
this yields an exact sequence
\begin{align*}
 0\to N^-\to P_\DK^\CJ(\mu)^{-}
\to \Delta_\DK(\mu)^-=\rDelta_\DK(\mu)\to 0
\end{align*}
in the category 
$\CO_{\DK,\crit}^{-\CJ}$ and each module admits a restricted Verma flag. 
Note that
$P_\DK^\CJ(\mu)^-$ is projective in $\CO_{\DK,\crit}^{-\CJ}$.

The above exact sequence yields the exact sequence
\begin{align*}
\Ext^{i-1}(N^-, \nabla_\DK(\lambda) )
\rightarrow
\Ext^i(\rDelta_\DK(\mu), \nabla_\DK(\lambda) )
\rightarrow
\Ext^i(P_\DK^\CJ(\mu)^{-}, \nabla_\DK(\lambda))(=0).
\end{align*}
But 
$\Ext^{i-1}(N^-, \nabla_\DK(\lambda) )=0$
by the induction hypothesis, so we can deduce
 $\Ext^i(\rDelta_\DK(\mu), \nabla_\DK(\lambda) )=0$. This  finishes the proof that $(1)$ implies $(2)$.

It is clear that $(2)$ implies $(3)$, and now we prove that $(3)$ implies $(1)$.
Let $X$ be as in $(3)$.
Let
$0=X_0\subset X_1\subset X_2
\dots \subset
X_l=X$
be a
highest weight series of $X$. That means each quotient $X_i/X_{i-1}$ is a highest weight module with highest weight $\mu_i$. We may also assume 
that
$\mu_i\not<\mu_j$
for any $i<j$.
We prove by induction on $l$
that
this sequence
is actually a restricted  Verma flag of $X$.

Let $l=1$.
Then,
$X=X_1$ is a highest weight module of highest weight $\mu:=\mu_1$. As it is contained in $\CO_{\DK,\crit}^{-\CJ}$ we have a surjection $\rDelta_\DK(\mu)\rightarrow X$. Let $N$ be its kernel.
We have to show that
$N=0$. This is equivalent to showing that $\Hom(N,\nabla_{\DK}(\lambda))=0$ for all $\lambda$.
From the exact sequence
$$0\rightarrow N\rightarrow \rDelta_\DK(\mu)\rightarrow X\rightarrow 0$$
we obtain, for any $\lambda$, the exact sequence
\begin{align*}
\Hom(\rDelta_\DK(\mu),\nabla_\DK(\lambda))
\rightarrow \Hom(N,\nabla_\DK(\lambda))
\rightarrow \Ext^1(X,\nabla_\DK(\lambda))(=0).
\end{align*}
But the space
$\Hom(\rDelta_\DK(\mu),\nabla_\DK(\lambda))$ vanishes unless $\lambda=\mu$
and for $\lambda=\mu$ the space  $\Hom(N,\nabla_\DK(\mu))$
is zero,
since there is no weight vector of weight $\mu$
in $N$ (here we need the assumption that our deformation algebra is a field). Hence $\Hom(N,\nabla_\DK(\lambda))=0$ for all $\lambda$, hence $N=0$ and $X\cong \rDelta_\DK(\mu)$.

Now let $l\geq 2$
and consider
the exact sequence
$$0\rightarrow X_1\rightarrow X\rightarrow X/X_1\rightarrow 0.$$
Using the induction hypothesis it is sufficient to show that
\begin{align*}
 \Ext^1(X_1,\nabla_\DK(\lambda))=\Ext^1(X/X_1,\nabla_\DK(\lambda))=0
\end{align*}
for all $\lambda$.
Consider the  long exact sequence
\begin{align*}
0
&\rightarrow \Hom(X/X_1,\nabla_\DK(\lambda))
\rightarrow \Hom(X,\nabla_\DK(\lambda))
\rightarrow \Hom(X_1,\nabla_\DK(\lambda))\\
&\rightarrow \Ext^1(X/X_1,\nabla_\DK(\lambda))
\rightarrow \Ext^1(X,\nabla_\DK(\lambda))(=0)
\rightarrow \Ext^1(X_1,\nabla_\DK(\lambda))\\
&\rightarrow \Ext^2(X/X_1,\nabla_\DK(\lambda))
\rightarrow \dots
\end{align*} 
Since $X_1$
is a quotient of $\rDelta_\DK(\mu_1)$,
it follows that $\Hom(X_1,\nabla_\DK(\lambda))=
0$
unless $\lambda=\mu_1$.
Hence,
$\Ext^1(X/X_1,\nabla_\DK(\lambda))=0$
unless $\lambda=\mu_1$. We now show that $\Ext^1(X/X_1,\nabla_\DK(\mu_1))=0$, so let 
$$
0\to \nabla_\DK(\mu_1)\to Y\to X/X_1\to 0
$$
be a short exact sequence in $\CO_{\DK,\crit}^{-\CJ}$ and consider the dual sequence
\begin{align*}
0\rightarrow (X/X_1)^\star \rightarrow Y^\star\rightarrow
\Delta_{\DK}(\mu_1)\rightarrow 0.
\end{align*}
Since $\mu_1\not<\mu_i$
for any $i\geq 2$, $\mu_1$ is a maximal weight of $Y$.
This means that the above sequences split,
proving that $\Ext^1(X/X_1,\nabla_\DK(\mu_1))=0$. 

Hence $\Ext^1(X/X_1,\nabla_\DK(\lambda))=0$
for any $\lambda\in\CJ$, so by our induction hypothesis $X/X_1$ admits a restricted Verma flag. By what we have already proven,
$\Ext^2(X/X_1,\nabla_\DK(\lambda))=0$ for all $\lambda\in\CJ$. From the long exact sequence above we deduce $\Ext^1(X_1,\nabla_\DK(\lambda))=0$ for all $\lambda\in\CJ$, hence $X_1$ admits a restricted Verma flag, hence so does $X$.
This completes the proof.
\end{proof}

\subsection{An $\Ext$-vanishing result}  We can now prove that, in the case that $A=\DK$ is a field, the module $\rP_\DK^\CJ(\lambda)$ admits a restricted Verma flag. For this we have to check an $\Ext$-vanishing property, by  Proposition \ref{prp:Ext-vs-Verma-flag-1}.

\begin{proposition}  One has $\Ext_{\CO_{\DK,\crit}^{-\CJ}}^1(\rP_\DK^\CJ(\lambda),\nabla_\DK(\mu))=0$
for all  $\lambda,\mu\in\CJ$.
\end{proposition}
\begin{proof} 
 Let 
\begin{align*}
 0\to \nabla_\DK(\mu)\to M\to \rP_\DK^\CJ(\lambda)
\to  0
\end{align*}
be an exact sequence  in $\CO_{\DK,\crit}^{-\CJ}$.
One needs to show that this sequence splits.

Because $\Delta_\DK(\mu)$ is a free $\CZ_{\crit}^{\le 0}$-module we have
$$
\Ext^1_{\CZ_{\crit}^{+}}(\DC,\nabla_\DK(\mu))
=\Ext^1_{\CZ_{\crit}^{-}}(\Delta_\DK(\mu),\DC)=0,
$$
where $\DC$ is the trivial $\CZ_{\crit}^{\pm}$-module.
Therefore, by applying the functor $\Hom_{\CZ_{\crit}^{+}}(\DC,?)$
to the above sequence we obtain the exact sequence
\begin{align*}
 0\to \rnabla_\DK(\mu)\to 
 \Hom_{\CZ_{\crit}^+}(\DC,M){\to}\rP_\DK^\CJ(\lambda)
\to  0
\end{align*}
in the category $\rCO_{\DK,\crit}^{\CJ}$. As the module on the right is projective in this category, we obtain a splitting $\rP_\DK^\CJ(\lambda)\to  \Hom_{\CZ_{\crit}^+}(\DC,M)$ and the composition with $\Hom_{\CZ_{\crit}^+}(\DC,M)\to M$ also splits our original short exact sequence.
\end{proof}

\begin{corollary}\label{Co:proj-admits-flag}
 Each $\rP_\DK^\CJ(\lambda)$ admits a 
restricted Verma flag.
\end{corollary}
Using Corollary \ref{Co:proj-admits-flag},
the following assertion can be proved in the same manner 
as
 Proposition \ref{prp:Ext-vs-Verma-flag-1}.

\begin{proposition}\label{prop:Ext-vs-Verma-flag}
Let $X\in \rCO_{\DK,\crit}^{\CJ}$.
The following conditions
are equivalent:
\begin{enumerate}
\item $X$ admits a restricted Verma flag. 
\item $X$ is finitely generated 
and $\Ext_{\rCO_{\DK,\crit}^{\CJ}}^i(X,
\rnabla_\DK(\lambda))=0$ for any $i\geq 1$ and  all $\lambda\in \CJ$.
\item $X$ is finitely generated 
and
$\Ext_{\rCO_{\DK,\crit}^{\CJ}}^1(X, \rnabla_\DK(\lambda) )=0$
for all $\lambda\in \CJ$.
\end{enumerate}
\end{proposition}

\subsection{A BGG-reciprocity formula} 
The next result allows us to compute restricted Verma multiplicities. 
\begin{proposition}\label{prop-vermamul} Suppose that $A$ is a local deformation algebra and that $M\in\CO_{A,\crit}^{-\CJ}$ admits a restricted Verma flag. For $\nu\in\CJ$ the following holds:
\begin{enumerate}
\item We have $\Ext^1_{\CO_{A,\crit}^{-\CJ}}(M,\rnabla_A(\nu))=0$.
\item $\Hom_{\CO_{A,\crit}^{-\CJ}}(M,\rnabla_A(\nu))$ is a free $A$-module of rank $(M:\rDelta_A(\nu))$.
\end{enumerate}
\end{proposition}
\begin{proof} The proof of part (1) is analogous to the proof of the corresponding statement in Proposition \ref{prp:Ext-vs-Verma-flag-1} (the field case), in particular, it can be analogously reduced to the case that $M\cong\rDelta_A(\lambda)$ for some $\lambda\in\CJ$. Consider a short exact sequence
$$
0\to \rnabla_A(\nu)\to X\to\rDelta_A(\lambda)\to 0.
$$ 
If $\nu\not>\lambda$, then this sequence splits by Lemma
\ref{lemma-projdomresVerma}. Each weight space in the above sequence is  a free $A$-module of finite rank, so the duality is involutive and exact on the above sequence. If $\nu>\lambda$, then the dual sequence
$$
0\to \rnabla_A(\lambda)\to X^\star\to\rDelta_A(\nu)\to 0
$$ 
splits. Hence $\Ext^1_{\CO_{A,\crit}^{-\CJ}}(\rDelta_A(\lambda),\rnabla_A(\nu))=0$.

Now let us prove part (2). Again we use induction on the length $l$ of a restricted Verma flag of $M$. Suppose that $M\cong\rDelta_A(\lambda)$. We have  
$$
\Hom(\rDelta_A(\lambda),\rnabla_A(\nu))=\Hom(\Delta_A(\lambda),\rnabla_A(\nu)).
$$ 
The latter space vanishes if $\lambda\ne\nu$ and it is free of rank 1 if $\lambda=\nu$ (by the statement that is dual to statement (2) in Lemma \ref{lemma-resVermafree}).  So suppose that $l>1$ and choose $M_1\subset M$ such that $M_1$ and $M/M_1$ are non-zero and admit restricted Verma flags. By (1) we have an exact sequence
$$
0\to\Hom(M/M_1,\rnabla_A(\nu))\to\Hom(M,\rnabla_A(\nu))\to\Hom(M_1,\rnabla_A(\nu))\to 0
$$
and part (2) follows from the induction hypothesis and the additivity of the multiplicities with respect to short exact sequences.
\end{proof}

Now we can prove a reciprocity statement for the restricted projectives in the field case.

\begin{theorem}\label{theorem-resBGG} Suppose that $A=\DK $ is a field. Let $\CJ\subset\hfhd$ be open and bounded and let $\lambda\in \CJ$ be critical. Then we have 
$$
(\rP_\DK ^{\CJ}(\lambda):\rDelta_\DK (\mu))=
\begin{cases}
[\rnabla_\DK (\mu):L_\DK (\lambda)],&\text{ if $\mu\in\CJ$} \\
0,&\text{ otherwise}.
\end{cases}
$$
\end{theorem}

\begin{proof} Clearly, $(\rP_\DK ^{\CJ}(\lambda):\rDelta_\DK (\mu))=0$ if $\mu\not\in\CJ$. So suppose that $\mu\in\CJ$. Using  Proposition \ref{prop-vermamul} we have
\begin{align*}
(\rP_\DK ^{\CJ}(\lambda):\rDelta_\DK (\mu)) &= \dim_\DK \Hom_{\CO_{\DK,\crit}^{-\CJ}}(\rP_\DK ^{\CJ}(\lambda),\rnabla_\DK (\mu)) \\ &= \dim_\DK \Hom_{\rCO_{\DK,\crit}^{\CJ}}(\rP_\DK ^{\CJ}(\lambda),\rnabla_\DK (\mu)) \\
&= [\rnabla_\DK (\mu):L_\DK (\lambda)].
\end{align*}
The last identity is a consequence of the fact that $\rP_\DK ^{\CJ}(\lambda)\to L_\DK(\lambda)$ is a projective cover in $\rCO_{\DK,\crit }^{\CJ}$. 
\end{proof}

\subsection{Base change}
Now we need the following variant of Lemma \ref{lemma-freeAmod}.

\begin{lemma}\label{lemma-rVfloverA} Let $A$ be a local deformation domain with residue field $\DK $ and quotient field $Q$. Suppose that $M\in\CO^-_{A,{\crit}}$ has the property that both $M\otimes_A \DK \in\CO^-_{\DK,{\crit}} $ and $M\otimes_A Q\in\CO^-_{Q,{\crit}}$ admit restricted Verma flags and that the multiplicities coincide, i.e.~ that for all $\mu\in\hfhd$ we have
$$
(M\otimes_A \DK :\rDelta_\DK (\mu))=(M\otimes_A Q:\rDelta_Q(\mu)).
$$
Then $M$ admits a restricted Verma flag with $(M:\rDelta_A(\mu))=(M\otimes_A \DK :\rDelta_\DK (\mu))$ for all $\mu\in\hfhd$. 
\end{lemma}

\begin{proof} Let $\mu\in\hfhd$. From the above equality of multiplicities we deduce that 
$$
\dim_\DK  M_\mu\otimes_A \DK =\dim_Q M_\mu\otimes_A Q.
$$
By Lemma \ref{lemma-freeAmod}, $M_\mu$ is a free $A$-module. In particular, the natural homomorphism $M\to M\otimes_A Q$ is injective.

Now let $\mu\in\hfhd$ be a maximal weight of $M$, let $v\in M_\mu$ be a preimage of a non-zero element $\ol v\in (M\otimes_ A\DK )_\mu$. Let $M_1\subset M$ be the $\hfg_A$-submodule generated by $v$. We have a surjective homomorphism $\rDelta_A(\mu)\to M_1$ that sends a generator of $\rDelta_A(\mu)$ to $v$, as $M_1\in\CO_{A,{\crit}}^-$. Now $M_1\otimes_A Q$ is generated by the non-zero vector $v\otimes 1$ and since $M\otimes_A Q$ admits a Verma flag and $\mu$ is maximal we have $M_1\otimes_A Q\cong \rDelta_Q(\mu)$. We deduce that the homomorphism $\rDelta_A(\mu)\to M_1$ is also injective, hence an isomorphism.

As $M_1\otimes_A \DK $ is generated by $\ol v$ and by maximality of $\mu$ we have that $M_1\otimes_A \DK \cong \rDelta_\DK (\mu)$. Moreover, $(M/M_1)\otimes_A \DK$ and $(M/M_1)\otimes_A Q$ admit restricted Verma flags with coinciding multiplicities.
Hence we can assume, by induction on the length of the Verma flags of $M\otimes_A \DK $ and $M\otimes_A Q$, that $M/M_1$ admits a restricted Verma flag. Hence so does $M$.
\end{proof}

\subsection{The case of a local deformation domain}

Now we have proved all relevant statements in the field case. Our next objective is to generalize them to the local  case. The following is an almost immediate consequence of the BGG-reciprocity we proved above.

\begin{lemma}\label{lemma-resproj} Suppose that the deformation algebra $A=\DK $ is a field. Let $\CJ\subset\hfhd$ be a
  bounded open subset and let  $P\in \CO_{\DK,\crit}^{\CJ}$ be projective. Then the module $P^{\res}$ 
  admits a restricted Verma flag and the multiplicities are given by the
  following formula: 
$$
(P^{\res}:\rDelta_\DK (\mu))=\sum_{n\ge 0}q(n)(P:\Delta_\DK (\mu-n\delta))
$$
for all $\mu\in\CJ$.
\end{lemma}

\begin{proof} 
We can assume that $P=P_\DK^{\CJ}(\lambda)$ for some $\lambda\in\CJ$, so $P^\res=\rP_A^\CJ(\lambda)$. By the reciprocity results in Theorem \ref{theorem-projobjinO} and Theorem \ref{theorem-resBGG}, equivalent to
$$
[\rnabla_\DK (\mu):L_\DK (\lambda)]=\sum_{n\ge 0} q(n)[\nabla_\DK (\mu-n\delta):L_\DK (\lambda)]
$$
which is statement (2) of Lemma \ref{lemma-resmul} in terms of the dual Verma modules.
\end{proof}

Now we can translate the results that we obtained so far to the case of a local deformation domain $A$. We denote by $\DK$ its residue field and by $Q$ its quotient field.

\begin{theorem}\label{thm-BGGcircproj} Suppose that $A$ is a local deformation domain. Let $\CJ\subset\hfhd$ be open and bounded and let $\lambda\in\CJ$ be critical. Then $\rP^{\CJ}_A(\lambda)$ admits a restricted Verma flag with multiplicities
$$
(\rP_A^{\CJ}(\lambda):\rDelta_A (\mu))=
\begin{cases}
[\rnabla_\DK (\mu):L_\DK (\lambda)],&\text{ if $\mu\in\CJ$} \\
0,&\text{ otherwise}.
\end{cases}
$$
 \end{theorem}
\begin{proof} Note that the functor $M\mapsto M^\res$ commutes with the base change functors $\cdot\otimes_A\DK$ and $\cdot\otimes_A Q$.  By Lemma \ref{lemma-resproj}, the restricted Verma multiplicities of $(P_A^\CJ(\lambda)\otimes_A \DK)^\res$ and of $(P_A^\CJ(\lambda)\otimes_A Q)^\res$ coincide, so by Lemma \ref{lemma-rVfloverA}, $P_A^\CJ(\lambda)^\res=\rP_A^\CJ(\lambda)$ admits a restricted Verma flag with the same multiplicities. Hence the statement follows from the BGG-reciprocity result for $\DK$, as $P_A^\CJ(\lambda)\otimes_A\DK\cong P_\DK^\CJ(\lambda)$.
\end{proof}

\section{The restricted linkage principle and the restricted block decomposition}

In this section we use the above BGG-reciprocity  to prove our main theorem, the restricted linkage principle:

\begin{theorem}\label{theorem-newlinkage} For all critical
  $\lambda,\mu\in\hfhd$ we have $[\rDelta(\lambda):L(\mu)]=0$
  if $\mu\not\in\hCW(\lambda).\lambda$.  
  \end{theorem}

Note that the above statement refers to the non-deformed objects (i.e.~ we have $A=\DC$ here). However, for its proof we need the deformation theory developed in the main body of this paper. So let $A$ be an arbitrary local deformation domain with residue field $\DK $. As a first step we study the restricted block decomposition.

\subsection{The restricted block decomposition}
Let $\hfhd_\crit$ be the set of critical weights in $\hfhd$ and 
let $\sim^{\res}_A$ be the relation on $\hfhd_\crit$ that is
generated by setting $\lambda\sim_A^{\res} \mu$ if there is some open subset
$\CJ\subset\hfhd_\crit$ such that $L_A(\mu)$ is isomorphic to a subquotient of $\rP_A^{\CJ}(\lambda)$. For an
equivalence class $\Lambda\in\hfhd_\crit/_{\textstyle \sim_A^{\res}}$ let
$\rCO_{A,\Lambda}\subset\rCO_{A,\crit}$ be the full subcategory that contains
all objects $M$ that have the property that if $L_A(\lambda)$ occurs as a subquotient of $M$, then $\lambda\in\Lambda$. Then well-known  arguments yield the following.

\begin{theorem} \label{theorem-resblockdecomp}The functor
\begin{align*}
\prod_{\Lambda\in\hfhd_\crit/_{\scriptstyle \sim_A^{\res}}}\rCO_{A,\Lambda} &\to \rCO_{A,\crit}\\
(M_\Lambda)&\mapsto \bigoplus M_\Lambda,
\end{align*}
is an equivalence of categories.
\end{theorem}

\subsection{Critical restricted equivalence classes}
 Let us denote by
$\ol{\cdot}\colon\hfhd\to\fhd$, $\lambda\mapsto \ol\lambda$, the map that is dual to the
inclusion $\fh\to\hfh=\fh\oplus\DC D\oplus \DC K$. Note that $\ol{\delta}=\ol\Lambda_0=0$. For any subset $\Lambda$ of $\hfhd$ we denote by $\ol\Lambda\subset\fhd$ its image.

Suppose that $\Lambda\in\hfhd_\crit/_{\textstyle{\sim_A^{\res}}}$ is a critical restricted equivalence class. We define the corresponding set of finite integral roots and the finite integral Weyl group by
\begin{align*}
R_A(\Lambda)&:=\{\alpha\in R\mid 2(\lambda+\rho,\alpha)_\DK \in\DZ(\alpha,\alpha)_\DK \text{ for all $\lambda\in\Lambda$}\},\\
\CW_A(\Lambda)&:=\langle s_\alpha\mid \alpha\in R_A(\Lambda)\rangle\subset \CW.
\end{align*}

\begin{lemma} Let $\Lambda\in\hfhd_\crit/_{\textstyle{\sim_A^{\res}}}$ be a critical restricted equivalence class. Then we have
$$
\ol\Lambda=\CW_A(\Lambda).\ol\lambda
$$
for all $\lambda\in\Lambda$.
\end{lemma}

\begin{proof} Let $\Lambda^\prime$ be the equivalence class under $\sim_A$ generated by $\sim_A^\res$ (note that $\sim_A^\res$ is finer than $\sim_A$). By the Kac--Kazhdan theorem, $\Lambda^\prime$ is the orbit of $\lambda$ under the group $\hCW_A(\Lambda^\prime)\times \DZ\delta$, so $\ol\Lambda=\ol{\Lambda^\prime}$ is the image of the $\hCW_A(\Lambda^\prime)$-orbit of $\lambda$. As $\Lambda$ is critical, $\hCW_A(\Lambda^\prime)$ is the affinization of $\CW_A(\Lambda)$, and the translations act by translating by a multiple of $\delta$. Hence, the image of $\Lambda$ in $\fhd$ coincides with the $\CW_A(\Lambda)$-orbit of $\ol\lambda$.
\end{proof}

\subsection{Generic and subgeneric equivalence classes} Now we define the two most basic cases for equivalence classes.

\begin{definition} Let $\Lambda\in\hfhd_\crit/_{\textstyle{\sim^{\res}_A}}$ be a critical restricted equivalence class. We call $\Lambda$
\begin{enumerate} 
\item {\em generic}, if $\ol{\Lambda}\subset\fhd$ contains exactly one element,
\item {\em subgeneric},   if $\ol{\Lambda}\subset\fhd$ contains exactly two elements. 
\end{enumerate}
\end{definition}

We call  $\lambda\in\hfhd_\crit$ {\em generic} ({\em
  subgeneric}, resp.) if it is contained in a generic (subgeneric,
resp.) equivalence
class.

Let $\Lambda$ be a critical restricted equivalence class and $\alpha\in R_A(\Lambda)$. Let $\lambda\in\Lambda$ and suppose that $s_{\alpha}.\lambda\ne \lambda$. Then we have $s_\alpha.\lambda>\lambda$ if and only if $s_{-\alpha+\delta}.\lambda<\lambda$. We define $\alpha\uparrow\lambda$ to be the element in the set $\{s_\alpha.\lambda, s_{-\alpha+\delta}.\lambda\}$ that is bigger than $\lambda$.

\subsection{A special deformation}

Let $\tS$ be the localization of $S$ at the maximal ideal $S\fh$. This is a local deformation domain with the obvious $S$-algebra structure. Its quotient field is $\DC=\tS/\tS\fh$ and the category $\CO_\DC$ is identified with the usual category $\CO$. For each prime ideal $\fp\subset \tS$ we denote by $\tS_{\fp}$ the localization of $\tS$ at $\fp$.  We let $\tQ=\tS_{(0)}$ be the quotient field of  $\tS$.

\begin{lemma} \label{lem-defequivclasses} Let $\fp\subset\tS$ be a prime ideal of height one and let $\Lambda\subset\hfhd_\crit$ be an equivalence class for $\sim_{\tS_\fp}^{\res}$.
\begin{enumerate} 
\item If $\alpha^\vee\not\in\fp$ for all $\alpha\in R$, then $\Lambda$ is generic.
\item If $\alpha^\vee\in \fp$ for some $\alpha\in R$, then $\Lambda$ is either generic or subgeneric. In both cases we have $R_{\tS_\fp}(\Lambda)\subset\{\alpha,-\alpha\}$. 
\end{enumerate}
\end{lemma}
\begin{proof} Let $\DK $ be the residue field of $\tS_\fp$. For any $\beta\in R$ we have $(\lambda+\tau,\beta)_\DK =(\lambda,\beta)_\DK +(\tau,\beta)_\DK $ with $(\lambda,\beta)_\DK\in\DC$ and $(\tau,\beta)_\DK\in \fh$. Hence we have $2(\lambda+\tau,\beta)_\DK\in\DZ(\beta,\beta)_\DK$ if and only if $2(\lambda,\beta)_\DK\in\DZ(\beta,\beta)_\DK$ and $(\tau,\beta)_\DK =0$. The latter equality implies $\beta=\pm\alpha$. From this we deduce both of the above statements.
\end{proof}

Part (1) of the following Theorem is a direct consequence of the corollary above  and Theorem 4.8 in \cite{Fr05} (which
states that a generic restricted Verma module is simple, see also \cite{Hay}) and part (2) is
a direct consequence of the above  and  the main result (Theorem 5.9) in \cite{AF08}, that calculates the Jordan--H\"older multiplicities in the subgeneric situations.

\begin{theorem}\label{theorem-subproj}  Let $\Lambda\in\hfhd_\crit/_{\textstyle{\sim_A^{\res}}}$ be a critical restricted equivalence class and fix $\lambda\in\Lambda$.   Let  $\CJ\subset \hfhd_\crit$ be open and bounded.
\begin{enumerate}
\item Suppose that $\Lambda$ is generic. Then 
$$
 \rP_A^{\CJ}(\lambda)\cong\rDelta_A(\lambda)
$$
if $\CJ$ contains $\lambda$.
\item Suppose that $\Lambda$ is subgeneric and suppose that $\ol\Lambda=\{\ol\lambda, s_{\alpha}.\ol\lambda\}$ for some $\alpha\in R$. Then there is a non-split short exact sequence
$$
0\to \rDelta_A(\alpha\uparrow\lambda)\to \rP_A^{\CJ}(\lambda) \to \rDelta_A(\lambda)\to 0
$$
if $\CJ$ contains $\lambda$ and $\alpha\uparrow\lambda$.
\end{enumerate}
\end{theorem}

\begin{corollary}\label{cor-strucequiv} Let $\Lambda\in\hfhd_\crit/_{\textstyle{\sim^{\res}_A}}$ be a critical restricted equivalence class.
\begin{enumerate}
\item If $\Lambda$ is generic, then $\Lambda$ contains only one element. 
\item If $\Lambda$ is subgeneric, then there is some $\alpha \in R(\Lambda)$ such that $\Lambda$ is an orbit under the action of the subgroup $\hCW_\alpha\subset\hCW$ that is generated by the reflections $s_{\alpha+n\delta}$ for $n\in\DZ$.
\end{enumerate}
\end{corollary}

\begin{proposition} \label{prop-defandequclasses} The equivalence relation $\sim^{\res}_\tS$ is the common refinement of all the relations $\sim_{\tS_\fp}^{\res}$ for prime ideals $\fp\subset\tS$ of height one, i.e.~ $\sim_\tS^{\res}$ is generated by the relations $\lambda\sim_\tS^{\res}\mu$ if there is a prime ideal $\fp\subset \tS$ of height one such that $\lambda\sim_{\tS_\fp}^{\res}\mu$.
\end{proposition}

\begin{proof} Let us denote by $\sim^\prime$ the common refinement of  the relations $\sim_{\tS_\fp}^{\res}$ for prime ideals of height one. It suffices to show that if $\lambda,\mu\in\hfhd$ are critical such that there is an open bounded subset $\CJ$ of $\hfhd_\crit$ and $(\rP_\tS^\CJ(\lambda):\rDelta_\tS(\mu))\ne 0$, then $\lambda\sim^\prime\mu$. 

Let us consider the object $\rP_\tS^\CJ(\lambda)\otimes_\tS \tQ$. It is an object in $\rCO_{\tQ,\crit}^\CJ$ and admits a restricted Verma flag. We are going to apply the decomposition result in Theorem \ref{theorem-resblockdecomp} for the categories $\rCO_{\tQ,\crit}$ and $\rCO_{\tS_\fp,\crit}$.

Let $\Lambda^\prime\subset\hfhd_\crit$ be the equivalence class under
$\sim^\prime$ that contains $\lambda$. As $\Lambda^\prime$ is a union
of equivalence classes for $\sim_\tQ^{\res}$ we can find a unique decomposition
$$
\rP_\tS^\CJ(\lambda)\otimes_\tS \tQ= X\oplus Y,
$$
where $X$ and $Y$ are objects in $\rCO_{\tQ,\crit}$ admitting a restricted Verma flag such that for all $\nu\in\hfhd$ we have
\begin{align*}
(X:\rDelta_\tQ(\nu))\ne 0 &\Rightarrow  \nu\in\Lambda^\prime,\\
(Y:\rDelta_\tQ(\nu))\ne 0 &\Rightarrow  \nu\not\in\Lambda^\prime. 
\end{align*}

Let $\fp\subset \tS$ be a prime ideal of height one. As $\sim^\prime$ is  coarser than $\sim^{\res}_{\tS_\fp}$, we deduce that the inclusion $\rP_\tS^\CJ(\lambda)\otimes_\tS \tS_{\fp}\to \rP_\tS^\CJ(\lambda)\otimes_\tS \tQ=X\oplus Y$ induces a direct sum decomposition
$$
\rP_\tS^\CJ(\lambda)\otimes_\tS \tS_{\fp}=\left(\rP_\tS^\CJ(\lambda)\otimes_\tS \tS_{\fp}\cap X\right)\oplus\left(\rP_\tS^\CJ(\lambda)\otimes_\tS \tS_{\fp}\cap Y\right).
$$

Now each weight space of $\rP_\tS^\CJ(\lambda)$ is a free $\tS$-module of finite rank and we deduce that 
$$
\rP_\tS^\CJ(\lambda)=\bigcap_{\fp} \rP_\tS^\CJ(\lambda)\otimes_\tS \tS_\fp,
$$
where the intersection is taken over all prime ideals of height one. Hence we get an induced decomposition
$$
\rP_\tS^\CJ(\lambda)=\left(\rP_\tS^\CJ(\lambda)\cap X\right)\oplus\left(\rP_\tS^\CJ(\lambda)\cap Y\right).
$$
As $\rP_\tS^\CJ(\lambda)$ is indecomposable, and since $X\ne 0$ (since the restricted Verma module $\rDelta_\tQ(\lambda)$ certainly occurs in $X$), we get $Y=0$, i.e.~ all restricted Verma subquotients of $\rP_\tS^\CJ(\lambda)$ have highest weights in $\Lambda^\prime$. Hence $\sim_\tS^{\res}=\sim^\prime$.
\end{proof}

Now we can prove our main result, Theorem \ref{theorem-newlinkage}.
\begin{proof} We show that $\lambda\sim^{\res}_\DC\mu$ implies $\mu\in\hCW(\lambda).\lambda$. Note that by definition we have $\sim^{\res}_\DC=\sim^{\res}_{\tS}$.By Proposition \ref{prop-defandequclasses} we have  that $\sim_{\tS}^{\res}$ is the common refinement of all $\sim_{\tS_\fp}^{\res}$.  Lemma \ref{lem-defequivclasses} shows that the equivalence classes of $\sim_{\tS_\fp}^{\res}$ are either generic or subgeneric. But those we determined in Corollary \ref{cor-strucequiv}: They contain either one element or are orbits under a certain subgroup $\hCW_\alpha$ of $\hCW(\lambda)$. Hence $\lambda$ and $\mu$ must be contained in a common $\hCW(\lambda)$-orbit.
\end{proof}

\end{document}